\documentclass{amsart}
\usepackage{amssymb}
\usepackage{color}
\usepackage{graphicx}
\theoremstyle{plain}
\newtheorem{theorem}{Theorem}[section]
\newtheorem{corollary}[theorem]{Corollary}

\newtheorem{proposition}[theorem]{Proposition}
\newtheorem{conjecture}[theorem]{Conjecture}
\newtheorem{lemma}[theorem]{Lemma}
\newtheorem{rem}[theorem]{Remark}

\numberwithin{equation}{section}
\def\op{\operatorname}

\newcommand{\N}{\mathbb{N}}

\newcommand{\C}{\mathbb{C}}
\newcommand{\R}{\mathbb{R}}

\makeatletter
\@namedef{subjclassname@2020}{%
  \textup{2020} Mathematics Subject Classification}
\makeatother

\begin{document}

\title{Some Properties of a Class of Sparse Polynomials}

\author{Karl Dilcher}
\address{Department of Mathematics and Statistics\\
         Dalhousie University\\
         Halifax, Nova Scotia, B3H 4R2, Canada}
\email{dilcher@mathstat.dal.ca}

\author{Maciej Ulas}
\address{Jagiellonian University, Faculty of Mathematics and Computer Science,
Institute of Mathematics, \L{}ojasiewicza 6, 30-348 Krak\'ow, Poland}
\email{Maciej.Ulas@im.uj.edu.pl}

\thanks{Research of the first author supported in part by the Natural Sciences and Engineering
        Research Council of Canada, Grant \# 145628481. Research of the second author supported by a grant of the Polish National Science Centre, no. UMO-2019/34/E/ST1/0009}

\date{}

\setcounter{equation}{0}

\begin{abstract}
We study an infinite class of sequences of sparse polynomials that have binomial
coefficients both as exponents and as coefficients. This generalizes a sequence
of sparse polynomials which arises in a natural way as graph theoretic
polynomials. After deriving some basic identities, we obtain properties
concerning monotonicity and log-concavity, as well as identities involving
derivatives. We also prove upper and lower bounds on the moduli of the zeros
of these polynomials.
\end{abstract}

\maketitle

\section{Introduction}\label{sec1}

A {\it sparse polynomial\/} (in one variable) is usually defined to be a
polynomial in which the number of nonzero coefficients is small compared with
its degree, where it often depends on the context what ``small" means. As far
as notation is concerned, sparse polynomials are normally given through its
nonzero coefficients:
\[
f(z) = \sum_{j=0}^n c_j z^{\alpha_j},\qquad
0\leq\alpha_0<\alpha_1<\cdots <\alpha_n,
\]
where $c_j\neq 0$ for $0\leq j\leq n$ and the integer sequence $(\alpha_j)$
often, but not necessarily, has greater than linear growth in $j$. Also, we
clearly have deg$(f)=\alpha_n$. Sparse polynomials are applied in computer
algebra, cryptography, approximation and interpolation, and various areas of
applied mathematics.

The starting point of this paper is a specific sequence of sparse polynomials
which arises naturally from a graph theoretic question related to the expected
number of independent sets of a graph \cite{BDM2}. These polynomials are defined
by
\begin{equation}\label{1.1}
f_n(z) := \sum_{j=0}^{n}\binom{n}{j}z^{j(j-1)/2}.
\end{equation}
Various properties, including asymptotics, zero distribution, and number
theoretic properties, can be found in \cite{BDM1}--\cite{GN}.

In this paper we extend the polynomials in \eqref{1.1} to the class of
polynomials
\begin{equation}\label{1.2}
f_{m,n}(z) := \sum_{j=0}^{n}\binom{n}{j}z^{\binom{j}{m}},
\end{equation}
where we will usually consider the integer $m\geq 1$ as a fixed parameter, and
then study the sequence $(f_{m,n}(z))_{n\geq 0}$. It is clear from \eqref{1.2}
that $f_{1,n}(z)=(1+z)^n$ and $f_{2,n}(z)=f_n(z)$. At the other extreme, we have
$f_{m,n}(z)=2^n$ when $n<m$, and for all $m\geq 1$,
\begin{equation}\label{1.2a}
f_{m,m}(z) = z+2^m-1,\quad f_{m,m+1}(z)=z^{m+1}+(m+1)z+\big(2^{m+1}-m-2\big).
\end{equation}
Furthermore, we have
\begin{equation}\label{1.2b}
f_{m,n}(0) = \sum_{j=0}^{m-1}\binom{n}{j},\qquad f_{m,n}(1)=2^n.
\end{equation}

Some of the more basic results on the polynomials $f_{m,n}(z)$ can be stated
and proved in greater generality, which we will do whenever possible and
reasonable. More precisely, let ${\bf h}=(h_j)_{j\geq 0}$ be an eventually
increasing sequence of nonnegative integers. Then for each integer $n\geq 0$
we define the polynomial
\begin{equation}\label{1.3}
H_{n}^{\bf h}(z) = \sum_{j=0}^{n}\binom{n}{j}z^{h_{j}}.
\end{equation}
For each $m\geq 1$ we obviously have
\[
H_{n}^{\bf h}(z) = f_{m,n}(z)\quad\hbox{when}\quad h_{j} = \binom{j}{m}.
\]

In Section~2 we will derive various basic properties of these polynomials,
usually also specialized to the polynomials $f_{m,n}(z)$. Section~3 is then
devoted to monotonicity properties of the sequence
$(H_{n}^{\bf h}(z))_{n\geq 0}$ and to the log-concavity of $(f_{m,n}(z))_n$.
In Section~4 we find some
identities connecting the polynomials $f_{m,n}(z)$ with their derivatives, and
also obtain a partial differential equation satisfied by their generating
function. In Section~5 we derive bounds on the moduli of the roots of
$f_{m,n}(z)$, and actually prove somewhat more general results. We conclude
this paper with some remarks and conjectures on the real roots of $f_{m,n}(z)$.

\section{Some basic identities}

We begin this section with an identity connecting two infinite series. It can
also be seen as a generating function of the sequence $H_{n}^{\bf h}(z)$.

\begin{proposition}\label{prop:2.1}
Let ${\bf h}$ be an eventually increasing sequence of nonnegative integers.
Then for $|z|<1$ and $|t|\leq 1/2$ we have
\begin{equation}\label{2.1}
\sum_{n=0}^{\infty}H_{n}^{\bf h}(z)t^{n}
=\frac{1}{1-t}\sum_{j=0}^{\infty}\left(\frac{t}{1-t}\right)^{j}z^{h_{j}}.
\end{equation}
\end{proposition}

\begin{proof}
For $|z|<1$ the definition \eqref{1.3} shows that $|H_{n}^{\bf h}(z)|<2^n$.
If we then substitute \eqref{1.3} on the left of \eqref{2.1}, the following
change in the order of summation can be justified for $|t|\leq 1/2$:
\begin{align*}
\sum_{n=0}^{\infty}H_{n}^{\bf h}(z)t^{n}
&=\sum_{n=0}^{\infty}\sum_{j=0}^{n}\binom{n}{j}z^{h_{j}}t^{n}
=\sum_{j=0}^{\infty}z^{h_{j}}\sum_{n=j}^{\infty}\binom{n}{j}t^{n}\\
&=\sum_{j=0}^{\infty}t^{j}z^{h_{j}}
\left(\sum_{n=0}^{\infty}\binom{n+j}{j}t^{n}\right).
\end{align*}
Now the sum in large parentheses in the last term has the well-known evaluation
$1/(1-t)^{j+1}$, which immediately leads to the right-hand side of \eqref{2.1}.
\end{proof}

As a consequence of Proposition~\ref{prop:2.1} we get the following generating
function.

\begin{corollary}\label{cor:2.2}
For any integer $m\geq 1$ and for variables $z,t$ with $|z|<1$ and
$|t|\leq 1/2$ we have
\begin{equation}\label{2.2}
\sum_{n=0}^{\infty}f_{m,n}(z)t^{n}
=\frac{1}{1-t}\sum_{j=0}^{\infty}\left(\frac{t}{1-t}\right)^{j}z^{\binom{j}{m}}.
\end{equation}
\end{corollary}

For $m=1$, both sides are geometric series that are easily seen to be equal.
For $m=2$, the identity \eqref{2.2} reduces to Lemma~2.1 in \cite{BDM1}, which
was used to obtain certain identities for theta functions.

If we set $t=1/2$ in \eqref{2.1}, we get the following identity.

\begin{corollary}\label{cor:2.3}
For a sequence ${\bf h}$ as above and $|z|<1$, we have
\begin{equation}\label{2.3}
\sum_{j=0}^{\infty}z^{h_{j}}
=\sum_{n=0}^{\infty}\frac{1}{2^{n+1}}H^{\bf h}_{n}(z).
\end{equation}
\end{corollary}

When $h_j=2^j$ for $j\geq 0$, then the left-hand side of \eqref{2.3}is a
well-known power series which has the unit circle as a natural boundary; see
\cite{Ma}. If we take $h_j=\binom{j}{2}$, we get Theorem~3.1 in \cite{BDM1},
namely
\[
\sum_{n=0}^{\infty}2^{-n}f_n(q^2) = 2 + q^{-1/4}\theta_2(q)\quad (|q|<1),
\]
where $f_n(z)$ is defined by \eqref{1.1} and $\theta_2(q)$ is the Jacobi theta
function
\[
\theta_2(q)=\theta_2(0,q)=2\sum_{n=0}^{\infty}q^{(n+\frac{1}{2})^2}
=2q^{1/4}\sum_{n=1}^{\infty}q^{n(n-1)};
\]
see, e.g., \cite[Ch.~20]{DLMF}.

The next identity is, in a sense, a finite analogue of
Proposition~\ref{prop:2.1}.

\begin{proposition}\label{prop:2.4}
For any integer $n\geq 0$ and complex variables $t$ and $z$ we have
\begin{equation}\label{2.4}
\sum_{k=0}^{n}\binom{n}{k}t^{k}H^{\bf h}_{k}(z)
=\sum_{j=0}^{n}\binom{n}{j}t^j(1+t)^{n-j}z^{h_{j}}.
\end{equation}
\end{proposition}

\begin{proof}
We use the definition \eqref{1.3} and change the order of summation; then the
left-hand side of \eqref{2.4} becomes
\[
\sum_{k=0}^{n}\binom{n}{k}t^{k}H^{\bf h}_{k}(z)
=\sum_{k=0}^{n}\binom{n}{k}t^{k}\sum_{j=0}^{k}\binom{k}{j}z^{h_{j}}
=\sum_{j=0}^{n}\left(\sum_{k=j}^{n}\binom{n}{k}\binom{k}{j}t^{k}\right)z^{h_{j}}.
\]
The inner sum on the right has the known evaluation
$\binom{n}{j}t^{j}(1+t)^{n-j}$; see, e.g., \cite[Eq.~(3.117)]{Go}. This
completes the proof.
\end{proof}

If we set $t=-1/2$ in \eqref{2.4} and compare the right-hand side with
\eqref{1.3}, we get the following consequence.

\begin{corollary}\label{cor:2.5}
Let $n\geq 0$ be an integer. If $h_j$ has the same parity as $j$ for all
integers $0\leq j\leq n$, then
\begin{equation}\label{2.5}
H^{\bf h}_{n}(-z)
=2^n\sum_{k=0}^{n}\binom{n}{k}\big(-\tfrac{1}{2}\big)^{k}H^{\bf h}_{k}(z).
\end{equation}
\end{corollary}

To motivate our next result, we set $t=-1$ in \eqref{2.4}. Then on the
right-hand side only the term with $j=n$ is nonzero, and we get
\begin{equation}\label{2.6}
\sum_{k=0}^{n}\binom{n}{k}(-1)^{k}H^{\bf h}_{k}(z)
= (-1)^nz^{h_{n}}.
\end{equation}
If we consider \eqref{1.3} as the binomial transformation of
the sequence $z^{h_n}$, then \eqref{2.6} can, in fact, be seen as the inverse
transformation.
Next, by taking the derivative with respect to $t$ of both sides of \eqref{2.4}
and then setting $t=-1$, we get
\begin{equation}\label{2.7}
\sum_{k=0}^{n}\binom{n}{k}(-1)^{k}kH^{\bf h}_{k}(z)
= (-1)^nn\left(z^{h_{n}}+z^{h_{n-1}}\right).
\end{equation}
These are the two smallest cases of the following result.

\begin{proposition}\label{prop:2.6}
Let $\nu\geq 0$ be an integer. Then for all integers $n\geq 0$ we have
\begin{equation}\label{2.8}
\sum_{k=0}^{n}\binom{n}{k}(-1)^{k}k^{\nu}H^{\bf h}_{k}(z)
=\sum_{i=0}^{\nu}a_{i,\nu}(n)z^{h_{n-i}},
\end{equation}
where $a_{i,\nu}, i=0, 1, \ldots, \nu,$ are polynomials in $n$ of degree $\nu$
and with integer coefficients.
\end{proposition}

We see immediately that \eqref{2.6} and \eqref{2.7} are the special cases
$\nu=0,1$ of \eqref{2.8}. A more explicit form of the right-hand side of
\eqref{2.8} will appear in its proof. For the proof we require the following
lemma.

\begin{lemma}\label{lem:2.7}
Let $\nu\geq 0$ be an integer. Then for all $n\geq 1$ we have
\begin{equation}\label{2.9}
\sum_{k=0}^n(-1)^k\binom{n}{k}\binom{k}{j}k^{\nu} = 0\quad\hbox{for}\quad
j=0, 1,\ldots, n-\nu-1.
\end{equation}
\end{lemma}

\begin{proof}
It is easy to verify that
$\binom{n}{k}\binom{k}{j}=\binom{n}{j}\binom{n-j}{k-j}$, so that
\begin{align*}
\sum_{k=0}^n(-1)^k\binom{n}{k}\binom{k}{j}i^k
&= \binom{n}{j}\sum_{k=j}^n(-1)^k\binom{n-j}{k-j}k^{\nu}\\
&= \binom{n}{j}\sum_{k=0}^{n-j}(-1)^{j+k}\binom{n-j}{k}(k+j)^{\nu}\\
&= (-1)^j\binom{n}{j}\sum_{k=0}^{n-j}(-1)^k\binom{n-j}{k}
\sum_{i=0}^{\nu}\binom{\nu}{i}j^{\nu-i}k^\nu\\
&= (-1)^j\binom{n}{j}\sum_{i=0}^{\nu}\binom{\nu}{i}j^{\nu-i}
\left(\sum_{k=0}^{n-j}(-1)^k\binom{n-j}{k}k^i\right).
\end{align*}
The inner sum in this last expression is well known, and evaluates as 0
whenever $i<n-j$; see, e.g., \cite[Eq.~(1.13)]{Go}. Since $j\leq-\nu-1$ and
$i\leq \nu$, this condition is satisfied, which completes the proof of
\eqref{2.9}.
\end{proof}

\begin{proof}[Proof of Proposition~\ref{prop:2.6}]
With the definition \eqref{1.3}, the left-hand side of \eqref{2.8} can be
written as
\[
\sum_{k=0}^{n}\sum_{j=0}^{k}\binom{n}{k}\binom{k}{j}k^{\nu}(-1)^{k}z^{h_{j}}
=\sum_{j=0}^{n}\left(\sum_{k=0}^{n}\binom{n}{k}\binom{k}{j}k^{\nu}(-1)^{k}\right)z^{h_{j}},
\]
where we have used the fact that $\binom{k}{j}=0$ for $j>k$. The result now
follows from Lemma~\ref{lem:2.7}.
\end{proof}

We conclude this section with the following observation. We already remarked
that by \eqref{1.3} and \eqref{2.6}, the sequences $z^{h_{n}}$ and
$H_{n}^{\bf h}(z)$ are binomial transformations of each other. This also means
that we have
$$
\op{Span}(\{z^{h_{n}}:\;n\in\N\})=\op{Span}(\{H_{n}^{\bf h}(z):\;n\in\N\}),
$$
where as usual the span of a subset $\Phi$ of a vector space (here, the space
$\C[x]$) is the smallest linear subspace that contains $\Phi$.

\section{Monotonicity and log-concavity}

We recall that a sequence $(a_n)_{n\geq 0}$ of real numbers is called
{\it absolutely monotonic\/} if for all integers $r,n\geq 0$ we have
\begin{equation}\label{3.1}
\Delta^r a_n \geq 0,
\end{equation}
where $\Delta^r$ is the difference operator of order $r$, defined recursively
by $\Delta a_n=a_{n+1}-a_n$, $\Delta^0 a_n=a_n$, and
$\Delta^{r+1}=\Delta\circ\Delta^r$ for $r\geq 0$. It is well-known that
\begin{equation}\label{3.2}
\Delta^r a_n = \sum_{k=0}^r(-1)^r\binom{r}{k}a_{n+r-k},
\end{equation}
which is easy to see by induction. This also means that if $a_n=f(n)$, where
$f$ is a polynomial of degree $d$, then for $r>d$ we have $\Delta^r a_n=0$
for all $n\geq 0$.

As part of their study of the sequence of polynomials $f_n(z)$ defined by
\eqref{1.1}, Gawronski and Neuschel \cite[Theorem~3.2]{GN} showed that for all
real $z$ with $0<z<1$ the sequence $(f_n(z))_{n\geq 0}$ is absolutely monotonic.
The following result is a generalization of this.

\begin{proposition}\label{prop:3.1}
Let ${\bf h}=(h_n)_{n\geq 0}$ be an eventually increasing sequence of
nonnegative integers. Then for all integers $r,n\geq 0$ we have
\begin{equation}\label{3.3}
\sum_{k=0}^{r}(-1)^{k}\binom{r}{k}H^{\bf h}_{n+r-k}(z)
=\sum_{k=0}^{n}\binom{n}{k}z^{h_{k+r}}.
\end{equation}
In particular, the sequence $(H^{\bf h}_{n}(z))_{n\geq 0}$ is absolutely
monotonic when $z>0$.
\end{proposition}

\begin{proof}
Using the definition \eqref{1.3}, we rewrite the left-hand side of \eqref{3.3}
as
\begin{equation}\label{3.4}
\sum_{k=0}^{r}(-1)^{k}\binom{r}{k}\sum_{j=0}^{n+r-k}\binom{n+r-k}{j}z^{h_{j}}
=\sum_{j=0}^{n+r}\left(\sum_{k=0}^{r}(-1)^{k}\binom{r}{k}\binom{n+r-k}{j}\right)z^{h_{j}},
\end{equation}
where we have extended the range of $j$ by adding zero-terms. Now we observe
that, by \eqref{3.2}, the inner sum on the right of \eqref{3.4} is just
$\Delta^r\binom{n}{j}$, and $\binom{n}{j}$ is a polynomial in $n$ of degree $j$.
Hence, by the remark following \eqref{3.2}, this sum is 0 for $j<r$. When
$j\geq r$, this inner sum has the known evaluation $\binom{n}{j-r}$; see, e.g.,
\cite[Eq.~(3.49)]{Go}. So, altogether the left-hand side of \eqref{3.3}, with
\eqref{3.4}, becomes
\[
\sum_{j=r}^{n+r}\binom{n}{j-r}z^{h_{j}} = \sum_{k=0}^{n}\binom{n}{k}z^{h_{r+k}},
\]
which was to be shown.
\end{proof}

As an immediate consequence we get the following.

\begin{corollary}\label{cor:3.2}
For any integer $m\geq 1$ and real $z>0$, the sequence $(f_{m,n}(z))_{n\geq 0}$
is absolutely monotonic.
\end{corollary}

When $m=1$, this reduces to the fact that $(x^n)_{n\geq 0}$
is an absolutely monotonic sequence for any real $x>1$.

While absolute monotonicity can be considered an additive property of a
sequence, we will see that, at least in a certain interval, the sequence
$\big(f_{m,n}(z)\big)_n$ also satisfies the {\it multiplicative}
property of log-concavity. A sequence $(a_n)$ is said to be {\it log-concave}
if $a_n^2\geq a_{n-1}a_{n+1}$ for all $n$. In this connection we have the
following result, which we will prove later.

\begin{proposition}\label{prop:3.3}
Let $m\geq 2$ be an integer, and set
\begin{equation}\label{3.5}
F_{m,n}(z) := \frac{f_{m,n}(z)^2-f_{m,n-1}(z)f_{m,n+1}(z)}{1-z}.
\end{equation}
Then for $n\geq m-1$ we have
\begin{equation}\label{3.6}
F_{m,n}(z)>0 \qquad\hbox{for real}\quad z\geq 0.
\end{equation}
In particular, the sequence $\big(f_{m,n}(z)\big)_{n\geq m-1}$ is log-concave
for $0\leq z\leq 1$.
\end{proposition}

Since $f_{m,n}(1)=2^n$, the numerator on the right of \eqref{3.5} vanishes,
and therefore $F_{m,n}(z)$ is a polynomial, which also means that $z=1$ is
a removable singularity of this fraction.

Also, since $f_{m,n}(z)=2^n$ when $n\leq m-1$, it is clear that $F_{m,n}(z)=0$
for $n\leq m-2$. Furthermore, by \eqref{1.2a} and for $m\geq 2$, we have
$F_{m,m-1}(z) = 2^{m-2}$ and
\[
F_{m,m}(z) = 2^{m-1}\big(z^m+z^{m-1}+\cdots+z^2\big)
+\big(2^{m-1}-1\big)z+(m-2)2^{m-1}+1.
\]
We notice that there are no negative coefficients.
Surprisingly, this is true in general, as we shall prove in the next result.

\begin{proposition}\label{prop:3.4}
For all integers $m\geq 2$ and $n\geq 1$, the polynomial $F_{m,n}(z)$ has
no negative coefficients.
\end{proposition}

\begin{proof}
Since $F_{m,n}(z)=0$ for $n\leq m-2$ and $F_{m,m-1}(z)$ is a constant, we may
assume that $n\geq m$. With \eqref{1.2}, and using the same idea as that of a
Cauchy product, we get
\begin{align*}
f_{m,n}(z)^2 &= \sum_{\nu=0}^{2n}\left(\sum_{j=0}^{\nu}
\binom{n}{j}\binom{n}{\nu-j}z^{\binom{j}{m}+\binom{\nu-j}{m}}\right),\\
f_{m,n-1}(z)f_{m,n+1}(z) &= \sum_{\nu=0}^{2n}\left(\sum_{j=0}^{\nu}
\binom{n-1}{j}\binom{n+1}{\nu-j}z^{\binom{j}{m}+\binom{\nu-j}{m}}\right),
\end{align*}
so that with \eqref{3.5} we have
\begin{equation}\label{3.9a}
(1-z)F_{m,n}(z) = \sum_{\nu=0}^{2n}g_{\nu}(z),
\end{equation}
where
\begin{equation}\label{3.10a}
g_{\nu}(z):=\sum_{j=0}^{\nu}\left(\binom{n}{j}\binom{n}{\nu-j}
-\binom{n-1}{j}\binom{n+1}{\nu-j}\right)z^{\binom{j}{m}+\binom{\nu-j}{m}}.
\end{equation}
Using the Chu-Vandermonde convolution twice (see, e.g., \cite[p.~169]{GKP} or
\cite[Eq.~(3.1)]{Go}), we see that $g_{\nu}(1)=0$ for all $\nu\geq 0$. This also
implies that
\[
g_{\nu}(z) = 0\qquad\hbox{for}\quad 0\leq\nu\leq m-1
\]
since $\binom{j}{m}=\binom{\nu-j}{m}=0$ for all $j$ in this case.

To get a more symmetric object, we reverse the summation in \eqref{3.10a}
and add the result to \eqref{3.10a}, obtaining
\begin{equation}\label{3.11a}
2g_{\nu}(z):=\sum_{j=0}^{\nu}a_{\nu,j}z^{\binom{j}{m}+\binom{\nu-j}{m}},
\end{equation}
where
\begin{equation}\label{3.12a}
a_{\nu,j}:=2\binom{n}{j}\binom{n}{\nu-j}-\binom{n-1}{j}\binom{n+1}{\nu-j}
-\binom{n+1}{j}\binom{n-1}{\nu-j}.
\end{equation}
By easy manipulations of the relevant binomial coefficients, we get
\begin{align*}
\binom{n}{j}\binom{n}{\nu-j}-\binom{n-1}{j}\binom{n+1}{\nu-j}
&= \binom{n}{j}\binom{n}{\nu-j}\cdot\frac{(2n+1)j-n\nu}{n(n+1-\nu+j)},\\
\binom{n}{j}\binom{n}{\nu-j}-\binom{n+1}{j}\binom{n-1}{\nu-j}
&= \binom{n}{j}\binom{n}{\nu-j}\cdot\frac{-(2n+1)j+(n+1)\nu}{n(n+1-j)}.
\end{align*}
Adding these two identities, we get after some further straightforward
manipulations,
\begin{equation}\label{3.13a}
a_{\nu,j}=\binom{n}{j}\binom{n}{\nu-j}\cdot
\frac{2(2n+1)j(\nu-j)-(n+1)\nu(\nu-1)}{n(n+1-\nu+j)(n+1-j)}.
\end{equation}
We obviously have $a_{\nu,j}=a_{\nu,\nu-j}$, and the numerator on the right,
which we call $N(j)$, is an increasing function of $j$, for
$0\leq j\leq\lfloor\nu/2\rfloor$. We clearly have $N(0)<0$, and it is easy to
verify that $N(\lfloor\nu/2\rfloor)>0$. Therefore, by \eqref{3.13a}, the integer
coefficients $a_{\nu,j}$ are negative for $0\leq j\leq j_0$, and then nonegative
for $j_0+1\leq j\leq \lfloor\nu/2\rfloor$, where $j_0\geq 1$, and at most one
zero value can occur.

We also note that the exponents $\binom{j}{m}+\binom{\nu-j}{m}$ of $z$ decrease
as $j$ increases from 0 to $\lfloor\nu/2\rfloor$. Since $g_{\nu}(1)=0$, the
coefficients add up to 0, and this means that each $-z^{\alpha}$ can be paired
with some $z^{\beta}$, where $\alpha>\beta$. Since we have
\[
-z^{\alpha}+z^{\beta}
= (1-z)z^{\beta}\big(z^{\alpha-\beta-1}+z^{\alpha-\beta-2}+\cdots+1\big),
\]
the polynomials $g_{\nu}(z)/(1-z)$, $\nu=0,1,\ldots,2n$, have no negative
coefficients, and so by \eqref{3.9a}, $F_{m,n}(z)$ has no negative coefficients
either. This completes the proof of Proposition~\ref{prop:3.4}.
\end{proof}

While much of Proposition~\ref{prop:3.3} is obvious from
Proposition~\ref{prop:3.4}, for the complete proof we need the following lemma,
which is also of independent interest. For positive integers $m$ and $n$ we
denote
\begin{equation}\label{3.14a}
S_{m}(n):=\left(\sum_{j=0}^{m-1}\binom{n}{j}\right)^2
-\left(\sum_{j=0}^{m-1}\binom{n-1}{j}\right)\cdot
\left(\sum_{j=0}^{m-1}\binom{n+1}{j}\right).
\end{equation}
By \eqref{3.5} and \eqref{1.2b} it is clear that $S_m(n)=F_{m,n}(0)$.

\begin{lemma}\label{lem:3.5}
For any integers $m\geq 2$ and $n\geq m-1$ we have $S_m(n)\geq 2^{m-2}$, with
equality when $n=m-1$.
\end{lemma}

\begin{proof}
We first note that for fixed $m$ the expression $S_{m}(n)$ can be seen as a
polynomial in $n$. We then use the fact that, by \eqref{1.2b}, we have
$S_{m}(j)=0$ for $j=1, 2, \ldots, m-2$, and thus
\[
\prod_{j=1}^{m-2}(n-j)\quad\hbox{divides}\quad S_{m}(n)
\]
as polynomials in $n$. With this in mind, we define the related polynomial
sequence
\begin{equation}\label{3.15a}
s_{m}(n):=\frac{(m-1)(m-2)!^{2}}{\prod_{j=1}^{m-2}(n-j)}S_{m}(n).
\end{equation}
One can check that the sequence $(s_{m}(n))_{m\geq 2}$ satisfies the recurrence
$s_{2}(n)=1$, $s_{3}(n)=n+2$, and
\begin{equation}\label{3.16a}
s_{m}(n)=(n+2)s_{m-1}(n)-(m-3)(n-m+2)s_{m-2}(n),\quad m\geq 4;
\end{equation}
to do so, we used the Zeilberger algorithm; see \cite[pp.~101--119]{PWZ}.
We now use the shifted sequence $g_{m}(t):=s_{m}(t+m-1)$ and show that the
polynomials $g_{m}(t)$ have only positive coefficients.

First, from \eqref{3.16a} it is clear that the sequence $(g_{m}(t))_{m\geq 2}$
satisfies the recurrence relation $g_{2}(t)=1, g_{3}(t)=t+4$, and
\begin{equation}\label{3.17a}
g_{m}(t)=(t+3m-5)g_{m-1}(t)-2(m-3)(t+m-2)g_{m-2}(t),\quad m\geq 4.
\end{equation}
From this relation it is clear that $\op{deg}g_{m}(t)=m-2$ and that the
coefficients of $g_{m}(t)$ are integers. We therefore write
\[
g_{m}(t)=\sum_{i=0}^{m-2}a_{i,m}t^{i},
\]
and from the recurrence \eqref{3.17a} we get that $a_{m-2,m}=1$,
$a_{0,m}=2^{m-2}(m-1)!$, and for $0\leq i\leq m-3$ we get the recurrence
\[
a_{i,m}=a_{i-1,m-1}-2(m-3)a_{i-1,m-2}+(3m-5)a_{i,m-1}-2(m-3)(m-2)a_{i,m-2}.
\]
Using this relation, we can prove by induction on $m$ that
\[
a_{i,m}>2(m-2)a_{i,m-1},\qquad 0\leq i\leq m-3,\quad m\geq 3.
\]
We thus get that $a_{i,m}>0$ for any $m\geq 2$ and $0\leq i\leq m-2$, and in
particular we have shown that $g_m(t)\geq 2^{m-2}(m-1)!$ for $t\geq 0$, or
$s_m(n)\geq 2^{m-2}(m-1)!$ for $n\geq m-1$. Finally, with \eqref{3.15a} we get
$S_m(n)\geq 2^{m-2}$, and the definition \eqref{3.14a} immediately gives
$S_m(m-1)=2^{m-2}$; this completes the proof.
\end{proof}

The proof of Proposition~\ref{prop:3.3} is now quite obvious:
Proposition~\ref{prop:3.4} shows that $F_{m,n}(z)$ is nonnegative for
$z\geq 0$, and Lemma~\ref{3.5} shows that it is in fact positive.

\begin{rem}\label{rem:3.6}
{\rm
(a) Computations suggest that for even $m\geq 2$, Proposition~\ref{prop:3.3} can
be extended to include the inequality $F_{m,n}(z)\geq 0$ for $-1\leq z\leq 0$.
In this case, Proposition~\ref{prop:3.4} will not be of much help since
$F_{m,n}(z)$ has both even and odd powers of $z$.

(b) One could also consider a natural extension of $F_{m,n}(z)$, namely
\begin{equation}\label{3.18a}
F_{m,n}^{(k)}(z) := \frac{f_{m,n}(z)^2-f_{m,n-k}(z)f_{m,n+k}(z)}{1-z},
\end{equation}
for arbitrary integers $k\geq 1$. Once again, by the right-hand identity in
\eqref{1.2b} the numerator vanishes for $z=1$, and therefore the quotient is a
polynomial. Computations suggest that the polynomial in \eqref{3.18a} has non-negative coefficients for any $k\geq 1$. We did not make a serious attempt at proving this. An approach as in the proof of Proposition~\ref{prop:3.4} might
work; however, the analogue of \eqref{3.13a} turns out to be difficult to deal
with for general $k$.
}
\end{rem}

In analogy to \eqref{3.18a} one could also consider the expression
\begin{equation}\label{3.19a}
G_{m,n}^{(k)}(z) := \frac{f_{m,2n}(z)-f_{m,n-k}(z)f_{m,n+k}(z)}{z-1},
\end{equation}
for arbitrary integers $k\geq 1$. By the right-hand identity in \eqref{1.2b},
this is again a polynomial which, moreover, has the following property.
We leave the proof to the interested reader.

\begin{proposition}\label{prop:3.7}
For all integers $m\geq 2$ and $n, k\geq 1$, the polynomial $G_{m,n}^{(k)}(z)$
has no negative coefficients.
\end{proposition}

We conclude this section with some observations based on computations in the
case $m=2$, i.e., the polynomials defined by \eqref{1.1}. If we iterate the
operator in \eqref{3.5}, then after each step
the resulting polynomial again seems to have no negative coefficients. We can
rephrase this in terms of log-concavity: Expanding on the definition just before
Proposition~\ref{3.3} we set
\[
{\mathcal L}(a_n) := a_n^2-a_{n-1}a_{n+1}
\]
for a given sequence $(a_n)$. Following \cite{MS}, the operator ${\mathcal L}$
can be iterated, and $(a_n)$ is said to be $k$-log-concave if
${\mathcal L}^j(a_n)\geq 0$ for all $j=0,1,\ldots,k$. If $(a_n)$ is
$k$-log-concave for all $k>0$, then it is said to be $\infty$-log-concave.

\begin{conjecture}\label{conj:3.8}
The sequence $\big(f_n(z)\big) = \big(f_{2,n}(z)\big)$ is $\infty$-log-concave
for all $0\leq z\leq 1$.
\end{conjecture}

Since we have seen that $1-z\mid{\mathcal L}(f_n(z))$, it is clear from the
definition of ${\mathcal L}$ that the $2^{k-1}$th power of $1-z$ divides
${\mathcal L}^k(f_n(z))$. However, computations indicate that for $n\geq k$,
the power of $1-z$ dividing ${\mathcal L}^k(f_n(z))$ is in fact $2^k-1$. 

In light of this (conjectural) property, we now rephrase our earlier 
experimental observations as follows.


\begin{conjecture}\label{conj:3.9}
For all positive integers $k$ and $n\geq k$, the expression 
\[
\frac{{\mathcal L}^k\big(f_{2,n}(z)\big)}{(1-z)^{2^{k}-1}}
\]
is a polynomial with positive integer coefficients.
\end{conjecture}

\section{Derivative properties}

In this section we establish
a connection between the polynomials $f_{m,n}(z)$ and their derivatives.
As a consequence we then obtain a partial differential equation satisfied by
the corresponding generating function. We begin with a consequence of
Proposition~\ref{prop:3.1}.

\begin{corollary}\label{cor:3.6}
Let $m\geq 1$ be an integer. Then for all $n\geq m$ we have
\begin{equation}\label{3.8}
zf'_{m,n}(z)=\binom{n}{m}\sum_{i=0}^{m}\binom{m}{i}(-1)^{i}f_{m,n-i}(z).
\end{equation}
\end{corollary}

\begin{proof}
Replacing $i$ by $k$ and $m$ by $r$, we see that the sum on the right of
\eqref{3.8} is just the left-hand side of \eqref{3.3}, with $h_k=\binom{k}{m}$.
On the other hand, we get from \eqref{1.2} by differentiation,
\begin{equation}\label{3.9}
\frac{zf_{m,n}'(z)}{\binom{n}{m}}
= \sum_{j=m}^n\frac{\binom{n}{j}\binom{j}{m}}{\binom{n}{m}}z^{\binom{j}{m}}.
\end{equation}
The three binomial coefficients on the right are easily seen to combine to
$\binom{n-m}{j-m}$, and with the appropriate changes in notation and shift
in summation, we see that the expression in \eqref{3.9} is equal to the
right-hand side of \eqref{3.3}.
\end{proof}

In the special case $m=2$ we get the following identity for the polynomials
defined in \eqref{1.1}: For $n\geq 2$ we have
\[
zf_n'(z) = \binom{n}{2}\big(f_{n}(z)-2f_{n-1}(z)+f_{n-2}(z)\big);
\]
This is Proposition~3.2 in \cite{BDM3}.
Next we use \eqref{3.8} to write
$f_{m,n}(z)$ in terms of derivatives.

\begin{corollary}\label{cor:3.7}
For integers $1\leq m\leq n$ we have
\begin{equation}\label{3.10}
f_{m,n}(z)=\sum_{i=0}^{m-1}\binom{n}{i}
+z\sum_{i=m}^{n}\frac{\binom{n-i+m-1}{m-1}}{\binom{i}{m}}f_{m,i}'(z).
\end{equation}
\end{corollary}

\begin{proof}
By taking the derivative of the definition \eqref{1.2}, we see that the second
term on the right of \eqref{3.10} becomes
\begin{equation}\label{3.11}
\sum_{i=m}^{n}\frac{\binom{n-i+m-1}{m-1}}{\binom{i}{m}}
\sum_{j=m}^i\binom{i}{j}\binom{j}{m}z^{\binom{j}{m}}
=\sum_{j=m}^n\left(\sum_{i=m}^{n}\frac{\binom{n-i+m-1}{m-1}}{\binom{i}{m}}
\binom{i}{j}\right)\binom{j}{m}z^{\binom{j}{m}},
\end{equation}
where we have extended the range of $j$ by including zero-terms. We now claim
that the inner sum on the right of \eqref{3.11} is equal to
$\binom{n}{j}/\binom{j}{m}$. Then \eqref{3.11} together with the first term
on the right of \eqref{3.10} is $f_{m,n}(z)$ by definition, which was to be
shown.

To prove the claim, we first note that it is trivially true when $n<j$ since
both sides of the identity vanish. We therefore assume $n\geq j\geq m\geq 1$.
By expanding all the binomial coefficients and rearranging the factorial terms,
we see upon shifting the summation that the identity in question is equivalent
to
\[
\sum_{i=0}^{n-m}\binom{n-i-1}{m-1}\binom{i}{j-m} = \binom{n}{j}.
\]
But this is a known identity; see, e.g., \cite[Eq.~(3.3)]{Go}. The proof is now
complete.
\end{proof}

We can use Corollary~\ref{cor:3.6} to show that for a fixed $m\geq 1$ the
ordinary generating function for the sequence $(f_{m,n}(z))_{n\geq 0}$, as
given in \eqref{2.2}, satisfies a linear partial differential equation of order
$m$. More precisely, using the notation
\begin{equation}\label{3.12}
F_{m}(z,t)=\sum_{n=0}^{\infty}f_{m,n}(z)t^{n},
\end{equation}
we have the following result.

\begin{corollary}\label{cor:3.8}
Let $m\geq 1$ be an integer, and let $|z|<1$ and $|t|<1/2$. Then
\begin{equation}\label{3.13}
z\frac{\partial F_{m}(z,t)}{\partial z}=(-t)^{m}\sum_{j=0}^{m}\frac{1}{j!}\binom{m}{j}(t-1)^{j}\frac{\partial^{j} F_{m}(z,t)}{\partial t^{j}}.
\end{equation}
\end{corollary}

\begin{proof}
By differentiating \eqref{3.12} $j$-times with respect to $t$ and multiplying
the result by $(-t)^m(t-1)^j$, we get
\begin{align*}
(-t)^{m}(t-1)^{j}&\frac{\partial^{j} F_{m}(z,t)}{\partial t^{j}}
= \sum_{i=0}^j\binom{j}{i}(-1)^{m-i}t^{j-i+m}
\sum_{n=j}^{\infty}\frac{n!}{(n-j)!}f_{m,n}(z)t^{n-j}\\
&= \sum_{i=0}^j\binom{j}{i}(-1)^{m-i}
\sum_{n=j+m-i}^{\infty}\frac{(n-m+i)!}{(n-m+i-j)!}f_{m,n-m+i}(z)t^{n}.
\end{align*}
Next we multiply both sides by $\binom{m}{j}/j!$ and sum over all
$j=0,1,\ldots,m$. We also use the identity
\[
\binom{m}{j}\binom{j}{i} = \binom{m}{i}\binom{m-i}{j-i},
\]
which is easy to verify. Then we get
\begin{align}
&(-t)^m\sum_{j=0}^m\frac{1}{j!}\binom{m}{j}(t-1)^{j}
\frac{\partial^{j}}{\partial t^{j}}F_{m}(z,t) \label{3.14}\\
&= \sum_{j=0}^m\binom{m}{j}\sum_{i=0}^j\binom{j}{i}\sum_{n=j+m-i}^{\infty}
\binom{n-m+i}{n-m+i-j}(-1)^{m-i}f_{m,n-m+i}(z)t^n\nonumber \\
&= \sum_{n=0}^{\infty}\sum_{i=0}^m\binom{m}{i}
\bigg(\sum_{j=i}^m\binom{m-i}{j-i}\binom{n-m+i}{n-m+i-j}\bigg)
(-1)^{m-i}f_{m,n-m+i}(z)t^n.\nonumber
\end{align}
Now the inner-most sum can be rewritten as
\begin{equation}\label{3.15}
\sum_{j=0}^{m-i}\binom{m-i}{j}\binom{n-(m-i)}{n-m-j} = \binom{n}{n-m},
\end{equation}
where we have used a known evaluation; see, e.g., \cite[Eq.~(3.4)]{Go}.

On the other hand, differentiating \eqref{3.12} with respect to $z$ and using
Corollary~\ref{cor:3.6}, we get
\begin{align*}
z\frac{\partial}{\partial z}F_{m}(z,t) &= \sum_{n=0}^{\infty}zf_{m,n}'(z)t^n\\
&= \sum_{n=0}^{\infty}\binom{n}{m}\sum_{i=0}^m\binom{m}{i}(-1)^if_{m,n-i}(z)t^n\\
&= \sum_{n=0}^{\infty}\binom{n}{m}\sum_{i=0}^m\binom{m}{i}(-1)^{m-i}if_{m,n-m+i}(z)t^n.
\end{align*}
Finally, combining this with \eqref{3.14} and \eqref{3.15}, we immediately get
the desired identity \eqref{3.13}, which completes the proof.
\end{proof}

\section{Bounds on the roots of $f_{m,n}(z)$}

In \cite{BDM3} it was shown that the roots of each polynomial $f_n(z)$, as
defined by \eqref{1.1}, all lie in the annulus
\begin{equation}\label{4.0}
\frac{2}{n} < |z| < 1+\frac{3}{n}\log{n},\qquad (n\geq 3).
\end{equation}
In this section we are going to extend this to the polynomials $f_{m,n}(z)$,
for all $m\geq 3$. As was done in \cite{BDM3}, we will prove more general
results for both the upper and the lower bounds. This section will end with
a few remarks and conjectures on real roots.

\begin{proposition}\label{prop:4.1}
Let $m\geq 3$ be an integer, and consider a polynomial
\begin{equation}\label{4.1}
f(z):=z^{\binom{n}{m}}+\sum_{k=0}^{n-1}a_{k}z^{b_{k}},
\end{equation}
where the numbers $a_{k}\in{\mathbb C}$ and integers $b_{k}\geq 0$ satisfy the
conditions
\begin{equation}\label{4.2}
|a_{k}|\leq \binom{n}{k},\quad 0\leq b_{k}\leq \frac{k}{n-m}\binom{n-1}{m}\quad\mbox{for}\quad k\in\{0,1,\ldots,n-1\}.
\end{equation}
Then for $n\geq 6m+1$ all roots of the polynomial $f(z)$ lie on the disc defined
by
\begin{equation}\label{4.3}
|z| < 1+\frac{m!}{(n-m)^{m-2}}.
\end{equation}
\end{proposition}

\begin{proof}
We follow the ideas used in \cite[Proposition 4.2]{BDM3}, modified as
necessary. More precisely, we wish to show that $|f(z)|>0$ for all $z$ with
sufficiently large modulus. To do so, we consider the number $(1+\varepsilon)z$
with $|z|=1$ and $\varepsilon>0$, which will be specified later. Using our
assumptions, we have the following chain of equations and inequalities:
\begin{align*}
|f((1+\varepsilon)z)|&=\left|((1+\varepsilon)z)^{\binom{n}{m}}
+\sum_{k=0}^{n-1}a_{k}((1+\varepsilon)z)^{b_{k}}\right|\\
&\geq (1+\varepsilon)^{\binom{n}{m}}-\sum_{k=0}^{n-1}a_{k}(1+\varepsilon)^{b_{k}}\\
&\geq (1+\varepsilon)^{\binom{n}{m}}-\sum_{k=0}^{n-1}\binom{n}{k}
\left((1+\varepsilon)^{\frac{1}{n-m}\binom{n-1}{m}}\right)^{k}\\
&=(1+\varepsilon)^{\binom{n}{m}}-\left[\left(1+(1+\varepsilon)^{\frac{1}{n-m}\binom{n-1}{m}}\right)^{n}
-(1+\varepsilon)^{\frac{n}{n-m}\binom{n-1}{m}}\right]\\
&=(1+\varepsilon)^{\binom{n}{m}}
-\left[\left(1+(1+\varepsilon)^{\frac{1}{n}\binom{n}{m}}\right)^{n}
-(1+\varepsilon)^{\binom{n}{m}}\right]\\
&=2(1+\varepsilon)^{\binom{n}{m}}
-\left(1+(1+\varepsilon)^{\frac{1}{n}\binom{n}{m}}\right)^{n},
\end{align*}
where in the penultimate equality we have used the identity
$\frac{1}{n-m}\binom{n-1}{m}=\frac{1}{n}\binom{n}{m}$. Thus, in order to get
the desired result it is enough to prove the inequality
\begin{equation}\label{4.3a}
2(1+\varepsilon)^{\binom{n}{m}}
>\left(1+(1+\varepsilon)^{\frac{1}{n}\binom{n}{m}}\right)^{n}.
\end{equation}
Raising both sides to the power $1/n$ and rewriting, we see that \eqref{4.3a} is
equivalent to
\begin{equation}\label{4.4}
(2^{\frac{1}{n}}-1)(1+\varepsilon)^{\frac{1}{n}\binom{n}{m}}>1.
\end{equation}
In \cite[Lemma 4.1]{BDM3} it was shown that for $x\in (0,2/5)$ we have
\begin{equation}\label{4.5}
\exp\left(\tfrac{5}{6}x\right)<1+x<\exp(x),
\end{equation}
which follows from the Maclaurin expansion of $\exp(x)$. Now we have
\[
2^{\frac{1}{n}}=\exp\left(\frac{1}{n}\log 2\right)>1+\frac{1}{n}\log 2
\]
from the right-hand inequality in \eqref{4.5}. From the left-hand inequality
of \eqref{4.5} we get
\[
(1+\varepsilon)^{\frac{1}{n}\binom{n}{m}}
>\exp\left(\frac{5\varepsilon}{6}\cdot\frac{1}{n}\binom{n}{m}\right).
\]
We thus see that in order to get the result it is enough to prove the inequality
\[
\exp\left(\frac{5\varepsilon}{6}\cdot\frac{1}{n}\binom{n}{m}\right)
\geq \frac{n}{\log 2},
\]
which by taking logarithms of both sides and solving for $\varepsilon$ is
equivalent to
\begin{equation}\label{4.6}
\varepsilon\geq \frac{6}{5}\cdot\frac{n(\log n-\log\log 2)}{\binom{n}{m}}
=:g_{m}(n).
\end{equation}
We observe that for $n\geq 2m+1$ we have $g_m(n)\leq \frac{2}{5}$. Indeed, we have
\[
\frac{1}{n}\binom{n}{m}\geq \frac{1}{n}\cdot\frac{n(n-m)^{m-1}}{m!}
=\frac{(n-m)^{m-1}}{m!}
\]
for $n\geq 2m+1$. Next, under the same assumption on $n$ we have
\[
\frac{6}{5}(\log n-\log\log 2)\leq \frac{n+1}{2}\leq n-m.
\]
From these inequalities we get
\[
g_{m}(n)=\frac{6}{5}\frac{\log n-\log\log 2}{\frac{1}{n}\binom{n}{m}}\leq \frac{n-m}{\frac{(n-m)^{m-1}}{m!}}=\frac{m!}{(n-m)^{m-2}}.
\]
By simple induction one can check that
\[
\frac{m!}{(n-m)^{m-2}}\leq \frac{2}{5}
\]
for $m\geq 6$ and $n\geq 2m+1$. If $m=3$ we need to take $n\geq 19$. For $m=4$
we take $n\geq 12$ and finally, if $m=5$, then $n\geq 12$ is also sufficient. As a
consequence, if we take $\varepsilon=g_{m}(n)$, then the inequality \eqref{4.4}
is satisfied and it holds also for all $\varepsilon>g_{m}(n)$ with $n\geq 6m+1$.
Finally, from our proof it is clear that
\[
g_{m}(n)\leq \frac{m!}{(n-m)^{m-2}}
\]
and thus if $m\geq 6$ and $n\geq 6m+1$, then all roots of the polynomial $f(z)$
satisfy the bound \eqref{4.3}.
\end{proof}

\begin{corollary}\label{cor:4.2}
For any integers $m\geq 3$ and $n\geq 2m+1$, all roots of the polynomial
$f_{m,n}(z)$ satisfy
\[
|z| < 1+\frac{m!}{(n-m)^{m-2}}.
\]
\end{corollary}

\begin{proof}
We apply Proposition~\ref{prop:4.1} with $a_k=\binom{n}{k}$ and
$b_k=\binom{k}{m}$. In order to verify \eqref{4.2}, we first rewrite
$\binom{k}{m}=\frac{k}{k-m}\binom{k-1}{m}$ and then note that for $k\leq n$
we have $\frac{1}{k-m}\binom{k-1}{m}\leq \frac{1}{n-m}\binom{n-1}{m}$, which
is easy to see.

Next, from the proof of Proposition~\ref{prop:4.1} we know that if $m\geq 6$
and $n\geq 2m+1$, then \eqref{4.3} holds. Finally, for $m=3, 4, 5$ and
$2m+1\leq n\leq 6m+1$, we verified the statement of the corollary by numerical
computations.
\end{proof}

For the lower bound, as well, we first prove a more general result. This is in
analogy to Proposition~4.4 in \cite{BDM3}.

\begin{proposition}\label{prop:4.3}
For a fixed integer $m\geq 3$ consider the polynomial
\begin{equation}\label{4.7}
G(z):= c_0+c_1z+\sum_{j=2}^{n-m+1}c_jz^{d_j},
\end{equation}
where $n\geq\lceil\sqrt[3]{2}\cdot m^{4/3}+m\rceil$ and $c_0, c_1,\ldots,c_{n-m+1}$
are real or complex coefficients satisfying
\begin{equation}\label{4.8}
|c_0|\geq\binom{n+1}{m-1},\qquad |c_j|\leq\binom{n}{m+j-1}\quad\hbox{for}\quad
1\leq j\leq n-m+1,
\end{equation}
and $d_2,\ldots,d_{n-m+1}$ are positive integers satisfying
\begin{equation}\label{4.9}
d_j\geq (m+1)(j-1),\qquad j=2, 3,\ldots, n-m+1.
\end{equation}
Then the zeros of $G(z)$ lie outside the circle about the origin with radius
$m/(n-m+1)$.
\end{proposition}

\begin{proof}
We use the identity $\binom{n+1}{m-1}=\binom{n}{m-2}+\binom{n}{m-1}$ and set
$r=|z|$. Then by the triangle inequality we have with \eqref{4.8} and
\eqref{4.9},
\begin{equation}\label{4.10}
|G(z)|\geq \binom{n}{m-2}+\binom{n}{m-1}-\binom{n}{m}r
-\sum_{j=2}^{n-m+1}\binom{n}{m+j-1}(r^{m+1})^{j-1},
\end{equation}
where we made the assumption that $r\leq 1$. Our goal is to show that when
$r\leq m/(n-m+1)$, then on the right of \eqref{4.10} the second and third terms
combined are nonnegative, and the sum on the right is strictly less than the
first term $\binom{n}{m-2}$.

First we note that $\binom{n}{m-1}-\binom{n}{m}r\geq 0$ is equivalent to
\[
r\leq \binom{n}{m-1}/\binom{n}{m} = \frac{n!m!(n-m)!}{(m-1)!(n-m+1)!n!}
= \frac{m}{n-m+1},
\]
as required. Next we rewrite
\begin{equation}\label{4.11}
\sum_{j=2}^{n-m+1}\binom{n}{m+j-1}(r^{m+1})^{j-1} = \binom{n}{m+1}r^{m+1}
\sum_{j=0}^{n-m-1}\frac{\binom{n}{m+j+1}}{\binom{n}{m+1}}(r^{m+1})^j.
\end{equation}
It is easy to verify that
\[
\binom{n}{m+j+1} \leq \binom{n-m-1}{j}\binom{n}{m+1}\qquad\Longleftrightarrow
\qquad 1 \leq \binom{m+j+1}{j},
\]
which is certainly true for all $j\geq 0$. Hence with \eqref{4.11} we get
\begin{align}
\sum_{j=2}^{n-m+1}\binom{n}{m+j-1}(r^{m+1})^{j-1}
&\leq \binom{n}{m+1}r^{m+1}\sum_{j=0}^{n-m-1}\binom{n-m-1}{j}(r^{m+1})^j\label{4.12}\\
&= \binom{n}{m+1}r^{m+1}\left(1+r^{m+1}\right)^{n-m-1}.\nonumber
\end{align}
Next we note that
\begin{align*}
\frac{1}{(1+r^{m+1})^{n-m-1}}
&> (1-r^{m+1})^{n-m-1} \geq 1-(n-m-1)r^{m+1},\\
&\geq 1-(n-m-1)r^4 \geq 1-(n-m-1)\frac{m^4}{(n-m+1)^4}
\end{align*}
for $m\geq 3$ and $r\leq m/(n-m-1)$. Hence
\[
\left(1+r^{m+1}\right)^{n-m-1}<2\quad\hbox{when}\quad
(n-m-1)\frac{m^4}{(n-m+1)^4}\leq\frac{1}{2}.
\]
But this holds when
\[
(n-m)\frac{m^4}{(n-m)^4}\leq\frac{1}{2}
\quad\Longleftrightarrow\quad n\geq\sqrt[3]{2}\cdot m^{4/3}+m,
\]
consistent with the hypotheses of the theorem. Under this condition we have
therefore with \eqref{4.12} that
\[
\sum_{j=2}^{n-m+1}\binom{n}{m+j-1}(r^{m+1})^{j-1} < 2\binom{n}{m+1}r^{m+1}
\leq 2\binom{n}{m+1}\frac{m^4}{(n-m+1)^4},
\]
again using $m\geq 3$ and $r\leq m/(n-m-1)$. Hence we are done if we can show
that
\[
2\binom{n}{m+1}\frac{m^4}{(n-m+1)^4} \leq \binom{n}{m-2}.
\]
Through expansion and cancellation it is easy to see that this inequality is
equivalent to
\begin{equation}\label{4.13}
(m-1)(m+1)(n-m+1)^3 \geq 2(n-m+2)(n-m)m^3.
\end{equation}
Since $(m^2-1)/m^2\geq\frac{8}{9}$ for $m\geq 3$ and
$(n-m+2)(n-m)=(n-m+1)^2-1$, we see that \eqref{4.13} holds when
\[
n-m+1\geq \frac{9}{4}m,\quad\hbox{or}\quad n\geq \frac{13}{4}m-1.
\]
Finally, it is easy to verify that
\[
\lceil\sqrt[3]{2}\cdot m^{4/3}+m\rceil\geq\lceil\tfrac{13}{4}m-1\rceil,\qquad
m=3,4,5,\ldots,
\]
with equality for $m=3, 4$, and 5. The proof is now complete.
\end{proof}

As a consequence of Proposition~\ref{prop:4.3} we get the following result,
which complements Corollary~\ref{cor:4.2}

\begin{corollary}\label{cor:4.4}
Let $m\geq 3$ and $n\geq\lceil\sqrt[3]{2}\cdot m^{4/3}+m\rceil$. Then the zeros of
$f_{m,n}(z)$ lie outside the circle centered at the origin with radius
$m/(n-m+1)$.
\end{corollary}

\begin{proof}
We rewrite \eqref{1.2} as
\begin{align}
f_{m,n}(z) &= 1+\binom{n}{1}+\cdots+\binom{n}{m-2}+\binom{n}{m-1}+\binom{n}{m}z\label{4.14}\\
&\quad +\sum_{j=2}^{n-m+1}\binom{n}{m+j-1}z^{\binom{m+j-1}{m}},\nonumber
\end{align}
and with
\[
c_0:=1+\binom{n}{1}+\cdots+\binom{n}{m-1},\quad
c_j:=\binom{n}{m+j-1},\quad j=1, 2,\ldots, n-m+1,
\]
we see that the condition \eqref{4.8} is satisfied. Next we prove by induction
that
\begin{equation}\label{4.15}
d_j:=\binom{m+j-1}{m} \geq (m+1)(j-1),\qquad j=2, 3, \ldots
\end{equation}
The induction beginning, for $j=2$, is obviously true as an identity. Suppose
now that \eqref{4.15} holds for some $j\geq 2$. Then
\begin{align*}
d_{j+1}&=\binom{m+j}{m}=\binom{m+j}{j}=\frac{m+j}{j}\binom{m+j-1}{j-1} \\
&=\frac{m+j}{j}\binom{m+j-1}{m} \geq \frac{m+j}{j}(m+1)(j-1),
\end{align*}
where the inequality at the end comes from the induction hypothesis. We are
done if we can show that
\begin{equation}\label{4.16}
\frac{m+j}{j}(j-1) \geq j\qquad\hbox{for}\quad j\geq 2 \quad\hbox{and}\quad
m\geq 3,
\end{equation}
since in that case \eqref{4.15} also holds for $j+1$. But \eqref{4.16} is
easily seen to be equivalent to $m(j-1)\geq j$, which is true for all $j\geq 2$
and $m\geq 3$. This proves \eqref{4.15}, which shows that the condition
\eqref{4.9} is also satisfied. Corollary~\ref{cor:4.4} therefore follows from
Proposition~\ref{prop:4.3}.
\end{proof}

We finish this section with a few observations and conjectures on the real
roots of the polynomials $f_{m,n}(z)$ which obviously have to be negative (if
there are any). Furthermore, by Corollary~\ref{cor:4.2}, the negative root with
largest absolute value $>1$ (if it exists) will be very close to $-1$, as $n$
grows and $m$ is fixed. To be more specific, we denote, for integers $m\geq 2$
and $n\geq 1$,
\[
N_m(n) = \left|\{z\in\R : f_{m,n}(z)=0\}\right|.
\]
It turns out that the behavior of the sequence $(N_m(n))_n$ is quite different
according as $m$ is even or odd.

\begin{conjecture}\label{conj:4.5}
If $m\geq 2$ is even, then $N_m(n)\leq\lfloor\frac{n}{m}\rfloor$ for all
$n\geq 1$, with equality in the case $m=2$.
\end{conjecture}

For the case $m=2$, this was earlier conjectured in \cite{BDM3}, along with more
detailed observations.

\begin{conjecture}\label{conj:4.6}
Let $m\geq 3$ be odd. Then for sufficiently large $n$,
\[
N_m(n)=\begin{cases}
2\;\hbox{or}\; 3 &\hbox{when}\;\;m\equiv 3\pmod{4},\\
m-1\;\hbox{or}\; m &\hbox{when}\;\;m\equiv 1\pmod{4}.
\end{cases}
\]
In both cases, the sequence $(N_m(n))_n$ is eventually periodic with period
$2^{\nu}$, where $2^{\nu-1}< m < 2^{\nu}$.
\end{conjecture}

Although a serious attempt at proving these conjectures would go beyond the
scope of this paper, we can make the following remarks.

1. The differences between even and odd $m$ can be partly explained by
considering the two highest powers of $z$ in $f_{n,m}(z)$, namely those with
exponents $\binom{n}{m}$ and $\binom{n-1}{m}$, as they would dominate the
behavior of the polynomial beyond, or close to, $z=-1$. In this connection one
can show that, when $m$ is odd, then $\binom{n}{m}$ and $\binom{n-1}{m}$ never
have the same parity, while in the opposite case they are both odd at least
once in each period of length $2^{\nu}$.

2. Applying methods similar to those used in the proofs in this section, it
should be possible to show the following for a fixed $m\geq 2$:

The polynomial $f_{n,m}(z)$ always has a root close to $-m/n$ for sufficiently
large $n$, and when $m$ is odd, there is another root close to
$(\frac{m+1}{n})^{1/m}$.

We obtained these heuristics by considering
only the constant and linear terms in $f_{n,m}(z)$ in the first case, and only
the linear and the following term (namely $\binom{n}{m+1}z^{m+1}$) in the
second case.

3. When $m$ is odd, it can be shown that $f_{m,n}(-1)>0$ for all $n$, which is
not the case for even $m$. It then follows that, when $\binom{n}{m}$ is odd,
there must be at least one more root to the left of $-1$. Since the parity of
$\binom{n}{m}$ is periodic with period $2^{\nu}$, this will be a partial
explanation of Conjecture~\ref{conj:4.6}.

\medskip
Much of the remarks made above depends on the behavior of the sequence
$(f_{m,n}(-1))_n$ and on the parities of the binomial coefficients
$\binom{j}{m}$. These arithmetic properties are studied, and applied in
different settings, in a forthcoming paper.

\end{document}